\documentclass[a4paper, 12pt]{amsart}

\usepackage[usenames,dvipsnames]{color}
\usepackage{amsthm,amsfonts,amssymb,amsmath,amsxtra}
\usepackage[all]{xy}
\SelectTips{cm}{}
\usepackage{xr-hyper}
\usepackage[colorlinks=
   citecolor=Black,
   linkcolor=Red,
   urlcolor=Blue]{hyperref}
\usepackage{verbatim}

\usepackage[margin=1.25in]{geometry}
\usepackage{mathrsfs}

\RequirePackage{xspace}
\RequirePackage{etoolbox}
\RequirePackage{varwidth}
\RequirePackage{enumitem}
\RequirePackage{tensor}
\RequirePackage{mathtools}
\RequirePackage{longtable}
\RequirePackage{multirow}

\def\ge{\geqslant}

\def\a{\alpha}

\def\d{\delta}

\def\e{\epsilon}

\def\s{\sigma}
\def\t{\tau}

\def\l{\lambda}

\def\i{^{-1}}

\def\<{\langle}
\def\>{\rangle}

\newcommand{\fkK}{\ensuremath{\mathfrak{K}}\xspace}

\newcommand{\fkR}{\ensuremath{\mathfrak{R}}\xspace}

\newcommand{\BC}{\ensuremath{\mathbb {C}}\xspace}

\newcommand{{\BG}}{\ensuremath{\mathbb {G}}\xspace}

\newcommand{{\BK}}{\ensuremath{\mathbb {K}}\xspace}

\newcommand{\BN}{\ensuremath{\mathbb {N}}\xspace}

\newcommand{\BQ}{\ensuremath{\mathbb {Q}}\xspace}
\newcommand{\BR}{\ensuremath{\mathbb {R}}\xspace}

\newcommand{\BZ}{\ensuremath{\mathbb {Z}}\xspace}

\newcommand{\CB}{\ensuremath{\mathcal {B}}\xspace}

\newcommand{\CI}{\ensuremath{\mathcal {I}}\xspace}

\newcommand{\CK}{\ensuremath{\mathcal {K}}\xspace}

\DeclareMathOperator{\End}{End}

\DeclareMathOperator{\Gal}{Gal}

\let\Im\relax
\DeclareMathOperator{\Im}{Im}

\DeclareMathOperator{\Tr}{Tr}

\newcommand{\red}{\ensuremath{\mathrm{red}}\xspace}

\newcommand{\rs}{\ensuremath{\mathrm{rs}}\xspace}

\DeclareMathOperator{\tr}{Tr}

\def\tW{\tilde W}

\def\cat{\mathsf{cat}}

\def\spe{\mathsf{sp}}

\def\el{\mathsf{ell}}

\def\rig{{\rm rig}}

\def\red{\mathsf{red}}

\def\good{\mathsf{good}}
\def\Rk{\mathsf{Rk}}

\newtheorem{theorem}{Theorem}
\newtheorem{proposition}[theorem]{Proposition}
\newtheorem{lemma}[theorem]{Lemma}
\newtheorem{theoremA}{Theorem}[subsection]

\newtheorem{corollary}[theorem]{Corollary}

\theoremstyle{definition}

\newtheorem{remark}[theorem]{Remark}

\numberwithin{equation}{section}
\numberwithin{theorem}{section}

\setitemize[0]{leftmargin=*,itemsep=\the\smallskipamount}
\setenumerate[0]{leftmargin=*,itemsep=\the\smallskipamount}

\renewcommand{\to}{%
   \ifbool{@display}{\longrightarrow}{\rightarrow}%
   }
\let\shortmapsto\mapsto
\renewcommand{\mapsto}{%
   \ifbool{@display}{\longmapsto}{\shortmapsto}%
   }
\newlength{\olen}
\newlength{\ulen}
\newlength{\xlen}
\newcommand{\xra}[2][]{%
   \ifbool{@display}%
      {\settowidth{\olen}{$\overset{#2}{\longrightarrow}$}%
       \settowidth{\ulen}{$\underset{#1}{\longrightarrow}$}%
       \settowidth{\xlen}{$\xrightarrow[#1]{#2}$}%
       \ifdimgreater{\olen}{\xlen}%
          {\underset{#1}{\overset{#2}{\longrightarrow}}}%
          {\ifdimgreater{\ulen}{\xlen}%
             {\underset{#1}{\overset{#2}{\longrightarrow}}}
             {\xrightarrow[#1]{#2}}}}%
      {\xrightarrow[#1]{#2}}
   }
\makeatother
\newcommand{\xyra}[2][]{%
   \settowidth{\xlen}{$\xrightarrow[#1]{#2}$}%
   \ifbool{@display}%
      {\settowidth{\olen}{$\overset{#2}{\longrightarrow}$}%
       \settowidth{\ulen}{$\underset{#1}{\longrightarrow}$}%
       \ifdimgreater{\olen}{\xlen}%
          {\mathrel{\xymatrix@M=.12ex@C=3.2ex{\ar[r]^-{#2}_-{#1} &}}}%
          {\ifdimgreater{\ulen}{\xlen}%
             {\mathrel{\xymatrix@M=.12ex@C=3.2ex{\ar[r]^-{#2}_-{#1} &}}}
             {\mathrel{\xymatrix@M=.12ex@C=\the\xlen{\ar[r]^-{#2}_-{#1} &}}}}}%
      {\mathrel{\xymatrix@M=.12ex@C=\the\xlen{\ar[r]^-{#2}_-{#1} &}}}%
   }
\makeatletter
\newcommand{\xla}[2][]{%
   \ifbool{@display}%
      {\settowidth{\olen}{$\overset{#2}{\longleftarrow}$}%
       \settowidth{\ulen}{$\underset{#1}{\longleftarrow}$}%
       \settowidth{\xlen}{$\xleftarrow[#1]{#2}$}%
       \ifdimgreater{\olen}{\xlen}%
          {\underset{#1}{\overset{#2}{\longleftarrow}}}%
          {\ifdimgreater{\ulen}{\xlen}%
             {\underset{#1}{\overset{#2}{\longleftarrow}}}
             {\xleftarrow[#1]{#2}}}}%
      {\xleftarrow[#1]{#2}}
   }
\newcommand{\isoarrow}{%
   \ifbool{@display}{\overset{\sim}{\longrightarrow}}{\xrightarrow\sim}%
   }
   
\begin{document}

\title{Cocenters of $p$-adic groups, III: Elliptic and rigid cocenters}

\date{\today}

\author{Dan Ciubotaru}
        \address[D. Ciubotaru]{Mathematical Institute, University of Oxford, Oxford, OX2 6GG, United Kingdom}
        \email{dan.ciubotaru@maths.ox.ac.uk}

\author{Xuhua He}
\address[X. He]{Department of Mathematics, University of Maryland, College Park, MD 20742 and Institute for Advanced Study, Princeton, NJ 08540, USA}
\email{xuhuahe@math.umd.edu}
\thanks{D.C. was partially supported by EPSRC EP/N033922/1. X. H. was partially supported by NSF DMS-1463852 and DMS-1128155 (from IAS)}

\begin{abstract}In this paper, we show that the elliptic cocenter of the Hecke algebra of a connected reductive $p$-adic group is contained in the rigid cocenter.  As applications, we prove the trace Paley-Wiener theorem and the abstract Selberg principle for mod-$l$ representations.
\end{abstract}

\maketitle

%\tableofcontents

\section*{Introduction}

\subsection{} Let $F$ be a nonarchimedean local field of residual characteristic $p$. Let $\BG$ be a connected reductive group over $F$ and $G=\BG(F)$ be the group of $F$-points. The study of smooth admissible representations of $G$ is a major topic in representation theory. In particular, the representation theory over complex numbers is a part of the local Langlands program, and it has been a central area of research in modern representation theory. The $G$-representations over an algebraically closed field $R$ of characteristic $l \neq p$ (in short, the mod-$l$ representations) is also a natural object to study, and it has attracted considerable interest recently due to the applications to number theory, e.g., congruences of the modular forms and the mod-$l$ Langlands program. Several important progresses have been achieved in this direction, e.g. Vign\'eras \cite{Vig,Vig3} and Vign\' eras-Waldspurger \cite{VW}. However, less is known compared to the complex representations. One of the major difficulty is that the proofs of many key results for complex representations rely heavily on  harmonic analysis methods, which are not always available for mod-$l$ representations. 

The main purpose of this paper (as well as of the previous papers \cite{hecke-1}, \cite{hecke-2}) is to develop a new approach towards the (complex and mod-$l$) representation theory of $G$. The approach is based on the relation between the cocenter $\bar H_R=H_R/[H_R, H_R]$ of the Hecke algebra $H_R=H(G)_R$ over the field $R$ and the Grothendieck group $\fkR_R(G)$ of representations over $R$ via the trace map $$\Tr_R: \bar H_R \to \fkR_R(G)^*.$$ A detailed analysis on the cocenter side should lead to a deep understanding on the representation side. It is also worth mentioning that the study of the structure of the cocenter is influenced by, and relies on, certain recent developments in arithmetic geometry,  in particular, the work the theory of $\s$-isocrystals and  affine Deligne-Lusztig varieties. 

\subsection{} We explain the main results of this paper. 

Let $M$ be a standard Levi subgroup. On the representations side, there are two important functors: the parabolic induction functor $i_M: \fkR_R(M) \to \fkR_R(G)$ and the Jacquet functor $r_M: \fkR_R(G) \to \fkR_R(M)$. On the cocenter side, there are also two important maps: the induction map $\bar i_M: \bar H(M) \to \bar H(G)$ adjoint to the Jacquet functor and the restriction map $\bar r_M: \bar H(G) \to \bar H(M)$ adjoint to the parabolic induction functor. 

An important family of irreducible representations is formed by the elliptic representations, which can be viewed as the set of irreducible smooth representations that are not linear combinations of induced representations $i_M(\s)$ for proper Levi subgroups (in the Grothendieck group). The analogous notion on the cocenter side is the elliptic cocenter $\bar H_R^{\el}$, the subspace of $\bar H$ consisting of elements $f$ such that $\bar r_M(f)=0 \in \bar H(M)$. 

Another important subspace of the cocenter $\bar H_R$ is the rigid cocenter $\bar H^{\rig}_R$.   By definition, $\bar H^{\rig}_R$ consists of elements in the cocenter $\bar H_R$ that are represented by functions with support in the subset of compact-modulo-center elements of $G$. As explained in \cite[Theorem C]{hecke-2}, the rigid cocenters form the ``building blocks'' of the whole cocenter $\bar H_R$. We will also see later in Theorem \ref{thmC} that there is a close relation between the rigid cocenter and the finitely generated projective representations of $G$. 

The main result of this paper compares these two important subspaces of the cocenter. The study of the relation between the elliptic cocenter and the rigid cocenter (as well as the duality with relevant spaces of representations) is also motivated by the results for affine Hecke algebras obtained in \cite{CH}.

\begin{theoremA}[Theorem \ref{main}]\label{thmA}
The elliptic cocenter is contained in the rigid cocenter. 
\end{theoremA}

As applications of Theorem \ref{thmA}, we will also establish two important results for mod-$l$ representations: the trace Paley-Wiener theorem and the abstract Selberg principle. 

\begin{theoremA}[Theorem \ref{trace-p-w}]\label{thmB}
Assume furthermore that the order of the relative Weyl group of $G$ is invertible in $R$. Then the image of the trace map $$\Tr_R: \bar H_R \to \fkR_R(G)^*$$ consists of all the good linear forms on $\fkR_R(G)$.
\end{theoremA}

Here ``good forms" is used in the sense of Bernstein-Deligne-Kazhdan \cite{BDK}. Roughly speaking, the set of irreducible representations of $G$ has an algebraic variety structure and the good forms correspond to the algebraic functions with respect to this structure. We refer to \S\ref{good} for the precise definition. 

\begin{theoremA}[Theorem \ref{Sel}]\label{thmC}
The rank map sends finitely generated projective representations of $G$ over $R$ into elements of the cocenter $\bar H_R$, which can be represented by functions with support in the subset of compact elements of $G$.
\end{theoremA}

The trace Paley-Wiener theorem was first established by Bernstein, Deligne and Kazhdan in \cite{BDK} for complex representations. It was later generalized to a twisted version by Flicker \cite{Fl} and more recently, by Henniart and Lemaire \cite{HL} for complex $\omega$-representations arising from the theory of twisted endoscopy. A key ingredient in these proofs is to establish a ``finiteness result'' for the discrete central characters via the study of tempered representations. 

The abstract Selberg principle was first established by Blanc and Brylinski in \cite{BB} for complex representations using low-dimensional cyclic and Hochschild homology and the Fourier transform. A different proof for $\text{char}(F)=0$ was given by Dat in \cite{Da2} using Clozel's integration formula and the  triviality of the action of unramified characters on the Grothendieck ring of projective modules of finite type. 

Both results were not known for mod-$l$ representations before. In the rest of the introduction, we will sketch our strategy to prove these results. 

\subsection{} Firstly, as we would like to develop a general machinery for arbitrary algebraically closed field of characteristic not equal to $p$, we work with the integral form $H$ of $H_R$, the algebra of $\BZ[\frac{1}{p}]$-valued functions, and its cocenter $\bar H=H/[H, H]$. 

The starting point is the Newton decomposition introduced in \cite{hecke-1}. Roughly speaking, we have decompositions $$G=\bigsqcup_{\nu \in \aleph} G(\nu) \quad \text{ and } \quad \bar H=\bigoplus_{\aleph} \bar H(\nu),$$ where $\aleph$ is the product of  $\pi_1(G)$ (the Kottwitz factor) and the set of dominant rational coweights of $G$ (the Newton factor), and $\bar H(\nu)$ is the $\BZ[\frac{1}{p}]$-submodule of $\bar H$ that can be represented by functions supported in $G(\nu)$. The subset of compact-modulo-center elements of $G$ is the union of $G(\nu)$ where the Newton factor of $\nu$ is central. The subset of compact elements of $G$ is the union of $G(\nu)$ where the Newton factor of $\nu$ is central and the Kottwitz factor of $\nu$ is of finite order. 

It is shown in \cite{hecke-2} that the Newton decomposition on $\bar H$ is compatible with the induction map. In section 2, we establish an explicit formula for the restriction map $\bar r_M: \bar H(G) \to \bar H(M)$, adjoint to the parabolic induction functor $i_M: \fkR_R(M) \to \fkR_R(G)$. In section 3, based on the explicit formula, we show that the restriction map $\bar r_M$ is compatible with the Newton decomposition of $\bar H(G)$ and $\bar H(M)$. This allows one to compute the image of $\bar r_M$ component-wise. In section 4, we prove a Mackey-type formula at the level of the cocenter. Note that in small characteristics, the trace map $\Tr_R: \bar H_R \to \fkR_R(G)^*$ is not injective and the desired formula for the cocenter does not follow from the Mackey formula for representations. Our proof of the Mackey-type formula for the cocenter is therefore rather involved. 

After we establish all these ingredients, we are able to compute $\bar r_M(f)$ for any $f \in \bar H(G)$ in a fairly explicit way and to show that if $\bar r_M(f)=0$ for all proper Levi subgroups $M$, then $f \in \oplus \bar H(\nu)$, where $\nu$ runs over elements in $\aleph$ with central Newton factor. Thus Theorem \ref{thmA} is proved. 

It is also worth mentioning that in the proof of Howe's conjecture in \cite{hecke-1}, one mainly uses only the fact that $\bar H$ is spanned by $\bar H(\nu)$; while in the proof of Theorem \ref{thmA} here, we use the full strength of the Newton decomposition $\bar H=\oplus \bar H(\nu)$. The fact that this is a direct sum decomposition allows us to do the component-wise computations and it plays a crucial way in the argument here. 

\subsection{} Now we discuss the strategy to prove the trace Paley-Wiener theorem. 

We first follow the combinatorial argument of \cite{BDK} and use the $A$-operator to reduce to the study of elliptic representations. This is the only part where we use the assumption that the order of the relative Weyl group is invertible in $R$. Then we establish the elliptic trace Paley-Wiener Theorem, based on a ``finiteness result''. 

The finiteness result in \cite{BDK} (as well as in \cite{Fl}, \cite{HL}) is that the set of discrete central characters in each given Bernstein component is a finite union of orbits under the action of the unramified characters. The finiteness result we establish here is different. We do not use the Bernstein components nor the irreducible representations, since for mod-$l$ representations, these objects are not well understood yet. We show instead that the image under the $A$-operators of the Grothendieck group of smooth admissible representations for a given depth is of finite rank.  This result will be deduced from the analogous result for the cocenter: the image in the cocenter $\bar H$ of the map $\bar A$ on the space of functions on $G$ of a given depth is finite dimensional. 

Note that it suffices to consider the depth $n-\e$, where $n$ is a positive integer and $\e>0$ is sufficiently small. In this case, one may just consider the functions in $H(G, \CI_n)$, the compactly supported $\CI_n$-biinvariant functions on $G$. Here $\CI_n$ is the $n$-th congruence subgroup of a fixed Iwahori subgroup $\CI$ of $G$. In this case, we have the Newton decomposition for depth $n-\e$ established in \cite{hecke-1}: $$\bar H(G, \CI_n)=\bigoplus_{\nu \in \aleph} \bar H(G, \CI_n; \nu).$$

The image under $\bar A$ of the cocenter $\bar H$ is exactly the elliptic cocenter that we discussed earlier. Now based on our main result Theorem \ref{thmA} and the Newton decomposition for depth $n-\e$, the image under $\bar A$ of $\bar H(G, \CI_n)$ is contained in the rigid cocenter of $\bar H(G, \CI_n)$, i.e. in $\bigoplus \bar H(G, \CI_n; \nu)$, where $\nu$ runs over elements in $\aleph$ with central Newton factor. Now the desired ``finiteness result'' follows from the fact that each Newton component $\bar H(G, \CI_n; \nu)$ is finite dimensional (\cite[Theorem 4.1]{hecke-1}). 

\subsection{} We now arrive at the abstract Selberg principle. 
For the purpose of the introduction, we assume that $G$ is semisimple. Then every compact-modulo-center element is compact and we need to show that the image of the rank map is contained in the rigid cocenter $\bar H^{\rig}_R$.

Let $f \in \bar H_R$ be the image of the rank map of a finitely generated projective module of $G$ over $R$. By the compatibility of the rank map with the Jacquet functor, and the inductive hypothesis, we have that $\bar r_M(f) \in \bar H(M)^\rig_R$. Then we deduce that $f \in \bar H_R^\rig$ by 

\begin{itemize}
\item the compatibility of the Newton decomposition on $\bar H(G)$ and on $\bar H(M)$;

\item a stronger form of Theorem \ref{thmA}, which says that for $\nu \in \aleph$ whose Newton factor is non-central, if the $\nu$-component of $\bar r_M(f)$ is trivial for all proper Levi $M$, then the $\nu$-component of $f$ is also trivial. 
\end{itemize}

\subsection{} Finally, in Appendix \ref{appendix}, we present a different proof (based on certain classical results of Clozel) of Theorem \ref{thmA} for Hecke algebras over $R$ under the assumptions that $F$ has characteristic $0$.

\section{Preliminaries}

\subsection{} We fix a minimal parabolic subgroup $P_0$ of $G$ and its Levi subgroup $M_0$. A parabolic subgroup is called {\it standard} if it contains $P_0$. A Levi subgroup $M$ is called {\it standard} if $M \supset M_0$ and it is a Levi subgroup of a standard parabolic subgroup. For any standard Levi subgroup $M$, we denote by $W_M$ its relative Weyl group. We choose the Haar measure $\mu_G$ on $G$ in such a way that the volume of the pro-$p$ Iwahori subgroup $\CI'$ is $1$. Then for any open compact subgroup $\CK$ of $G$, $\mu_G(\CK) \in \BZ[\frac{1}{p}]$. The Haar measures on the standard Levi subgroups are chosen in the same way. 

Let $H$ be the Hecke algebra of $G$ over $\BZ[\frac{1}{p}]$, i.e., the space of locally constant, compactly supported $\BZ[\frac{1}{p}]$-valued functions on $G$, endowed with convolution with respect to Haar measure $\mu$. %It is well known that such a normalization of the Haar measure exists (for example, by requiring the volume of the pro-unipotent radical of the maximal special compact subgroup of $G$ to be $1$).

Let $\bar H=H/[H, H]$ be the cocenter of $H$. If $f\in H$, denote by $\bar f$ the image of $f$ in $\bar H$.

We fix an algebraically closed field $R$ of characteristic not equal to $p$. Set \[H_R=H \otimes_{\BZ[\frac{1}{p}]} R \text{ and }\bar H_{R}=\bar H \otimes_{\BZ[\frac{1}{p}]} R.\]
Let $\fkR_R(G)$ be the $R$-vector space with basis $\text{Irr}H_R$, the isomorphism classes of irreducible smooth admissible $R$-representations of $G$. For every standard parabolic subgroup $P=M N$, denote by $\pi_P:P\to M$ the natural projection map. We also denote by $i_{M, R}: \fkR_R(M) \to \fkR_R(G)$ the functor of normalized induction and by $r_{M, R}: \fkR_R(G) \to \fkR_R(M)$ the normalized Jacquet functor, the functor of $N$-coinvariants. Whenever we speak of normalized induction/restriction, we assume tacitly that $p$ has a square root in $R$ and we fix such a square root (see also \cite[page 97]{Vig}).

\subsection{} 
Now we recall the categorical description of $\bar H_R$, see \cite[\S F]{Fl} and \cite[\S 5.1]{Da}. The cocenter of $H$ can be viewed as the quotient of the free abelian group on the symbols $(\pi, u)$, where $\pi$ is a finitely generated projective $H$-module over $R$, and $u \in \End_{H_R}(\pi)$ modulo the relations 

\begin{itemize}
\item[(C1)] $(\pi, u)=(\pi_1, u_1)+(\pi_2, u_2)$ for any short exact sequence $0 \to \pi_1 \to \pi \to \pi_2 \to 0$ that commutes with the $u$-action.

\item[(C2)] $(\pi, u)+(\pi, u')=(\pi, u+u')$.

\item[(C3)] $(\pi, f g)=(\pi', g f)$ if $f: \pi' \to \pi$ and $g: \pi \to \pi'$. 
\end{itemize}

As explained in \cite[\S 1.7]{Da}, the functors $i_{M, R}$ and $r_{M, R}$ send finitely generated projective modules to finitely generated projective modules. Therefore they induce maps on the symbols $(\pi, u)$, which preserve the above relations. We denote the induced map on the cocenter $\bar H_R$ by $$\bar i^{\cat}_{M, R}: \bar H(M)_R \to \bar H_R, \qquad \bar r^{\cat}_{M, R}: \bar H_R \to \bar H(M)_R.$$

By Frobenius reciprocity and Bernstein's second adjointness theorem, we have 
\begin{equation}\label{e:dual}
\begin{aligned}
\Tr^G_R(h_G, i_{M, R}(\s))&=\Tr^M_R(\bar r^{\cat}_{M, R}(h_G), \s), \\
\Tr^M_R(h_M, r_{M, R}(\pi))&=\Tr^G_R(\bar i^{\cat}_{M, R}(h_M), \pi),
\end{aligned}
\end{equation}
 for $h_G \in \bar H_R, h_M \in \bar H(M)_R, \pi \in \fkR_R(G), \s \in \fkR_R(M)$. 

\subsection{} 
For every smooth admissible $H_R$-representation $(\pi,V_\pi)$ and $f\in \bar H_R$, the operator $\pi(f)\in \End(V_\pi)$ is of trace class, and we denote by $\Tr_R(f,\pi)$ its trace. We have the induced trace map $$\Tr_R: \bar H_R \to \fkR_R(G)^*.$$

In the case where $R=\BC$, we have the trace Paley-Wiener theorem \cite{BDK} and spectral density theorem \cite{Kaz} which identifies the cocenter $\bar H_\BC$ with the the space of ``good'' linear forms on $\fkR_\BC(G)$ (we refer to \S \ref{good} for the precise definition of good forms). In this case we may regard $\bar i^{\cat}_{M, \BC}$ (resp. $\bar r^{\cat}_{M, \BC}$) as the dual functor of $r_{M, \BC}$ (resp. $i_{M, \BC}$). However, if $\text{char}(R)$ is small, then the trace map is not injective, and thus the relations about the functors on $\fkR_R(*)$ do not automatically imply the similar relations on $\bar H(*)_R$.  However, it is easy to see, using (\ref{e:dual}), that $\ker \Tr_R$ is stable under the induction and restriction maps. We define the {\it reduced cocenter} by $$\bar H_R^{\red}=\bar H_R/\ker \Tr_R.$$ For any open compact subgroup $\CK$, we define $\bar H_R(G, \CK)^{\red}$ to be the image of $\bar H_R(G, \CK)$ in $\bar H_R^{\red}$. 
\subsection{} We recall the Newton decomposition introduced in \cite{hecke-1}. 

Let $A$ be a maximal $F$-split torus in $M_0$ and $Z$ be the centralizer of $A$. We may identify the relative finite Weyl group $W_G$ with $N_G A(F)/Z(F)$. We fix a special vertex in the fundamental alcove of the apartment corresponding to $A$. The Iwahori-Weyl group can be realized as  $$\tW \cong X_*(Z)_{\Gal(\bar F/F)} \rtimes W_G=\{t^\l w; \l \in X_*(Z)_{\Gal(\bar F/F)}, w \in W_G\}.$$ 

We set $V=X_*(Z)_{\text{Gal}(\bar F/F)} \otimes \BR$ and $\Omega=\tW/W_a$, where $W_a$ is the affine Weyl group associated to $\tW$. Set $\aleph=\Omega \times V_+$, where $V_+$ is the set of dominant elements in $V$ with respect to the positive roots given by $P_0$. For any $\nu \in \aleph$, we have the corresponding Newton stratum $G(\nu)$. The precise definition is technical, and we refer to \cite[\S 2.2]{hecke-1} for more details. For any $\nu \in \aleph$, let $H(G; \nu)$ be the subspace of $H$ consisting of functions with support in $G(\nu)$ and $\bar H(G; \nu)$ be the image of $H(\nu)$ in $\bar H$. We may simply write $H(\nu)$ for $H(G; \nu)$ and $\bar H(\nu)$ for $\bar H(G; \nu)$. The following result is proved in \cite[Theorem 2.1  \& Theorem 3.1]{hecke-1}.

\begin{theorem}
We have the Newton decompositions $$G=\sqcup_{\nu \in \aleph} G(\nu), \quad H=\oplus_{\nu \in \aleph} H(\nu), \quad \bar H=\oplus_{\nu \in \aleph} \bar H(\nu).$$
\end{theorem}

In this paper, we are mainly interested in the $V$-factor of $\aleph$. For $v \in V_+$, we set $\bar H(v)=\oplus_{\nu=(\t, v) \text{ for some } \t \in \Omega} \bar H(\nu)$. 

Let $\bar H_R^{\red}(v)$ be the image of $\bar H_R(v)$ in $\bar H^{\red}$. The following result follows from \cite[Theorem 6.3]{hecke-2}.

\begin{theorem}\label{t:newton}
We have the Newton decomposition $$\bar H_R^{\red}=\oplus_{v \in V_+} \bar H_R^{\red}(v).$$
\end{theorem}
\begin{remark}
As explained in \cite[Remark 6.4]{hecke-2}, $\sum_{\nu \in \aleph} \bar H_R^{\red}(\nu)$ may not be a direct sum in general. 
\end{remark}

Following \cite[\S 6]{hecke-1}, we define 
\begin{equation}\label{e-rigid}
\bar H^{\rig}=\bigoplus_{v \in V_+; M_v=G} \bar H(v), \quad \bar H^{\rig, \red}_R=\bigoplus_{v \in V_+; M_v=G} \bar H^{\red}_R(v),
\end{equation}
where $M_v$ is the centralizer of $v$ in $G$. We call $\bar H^{\rig}$ the {\it rigid cocenter} and $\bar H^{\rig, \red}_R$ the reduced rigid cocenter. 

\section{The restriction map}

\subsection{} We give the formula for $\bar r_M$ which is the cocenter dual of van Dijk's formula \cite{Di} for parabolic induction of characters.

Let $\CK^{\spe}$ be the maximal special parahoric subgroup of $G$ corresponds to the special vertex of the fundamental alcove we fixed in the beginning. For any $w \in W_G$, we choose a representative ${\dot w}$ in $\CK^{\spe}$. 

Let $P=MN$ be a standard parabolic subgroup. In particular, the Iwasawa decomposition $G=\CK^{\spe}P$ holds. Fix a Haar measure $\mu_N$ on $N$. Let $\mu_P=\mu_M \mu_N$ be the Haar measure on $P=M N$. Let $\delta_P$ be the modulus function of $P$. 

Let $f \in H$ be given. The function $f$ gives rise to two functions: $f^{(P)}$, a function on $M$,  and $\widetilde f$, a function on $G$, defined as follows:
\begin{align*}
f^{(P)}(m)&=\d_P^{\frac{1}{2}}(m) \int_N  f(mn) dn, \text{ for } m\in M, \text{ and }\\ \widetilde f(g)&=\int_{\CK^{\spe}}f(k \i g k) dk, \qquad \text{ for } g \in G.
\end{align*}
It is immediate that 
\begin{equation}\label{P-hom}
(L_{p_0}f)^{(P)}=\delta_P^{\frac 12}(\pi_P(p_0)) L_{\pi_P(p_0)} f^{(P)},
\end{equation}
where $L_x$ denotes the left regular action. In other words, the map $f\mapsto f^{(P)}$ is a homomorphism of $P$-modules $H\to H(M)$, where the action on the $H$ is the left regular action, and the action on $H(M)$ is via $\delta_P^{\frac 12} \pi_P$. 

\begin{lemma}\label{l:well} The assignment $f\mapsto \widetilde f^{(P)}$ defines a linear map $H\to H(M)$.
\end{lemma}

\begin{proof}
It suffices to consider the case where $f=\d_X$, where $X$ is an open compact subset of $G$. Let $\CK' \subset \CK^{\spe}$ be an open compact subgroup such that $X$ is stable under the left and right multiplication of $\CK$. For any $m$, 
\begin{equation}\label{e:rP-explicit}
\widetilde f^{(P)}(m)=\mu(\CK')  \sum_{g \in \CK^{\spe}/\CK'} \sum_{m \in \pi_P({}^{g \i} X \cap P)} \d_P^{\frac{1}{2}}(m) \mu_N (({}^{g \i} X) m \i \cap N);
\end{equation}
recall that $\pi_P: P \to M$ is the natural projection map. Note that $\pi_P({}^{g \i} X \cap P)$ is an open compact subset of $M$, and the functions $\d_P^{\frac{1}{2}}(m)$ and $\mu_N (({}^{g \i} X) m \i \cap N)$ are locally constant $\BZ[\frac{1}{p}]$-valued functions on $M$. Thus $\widetilde f^{(P)}\in H(M)$.  
\end{proof}

\begin{lemma}\label{l:twist}
For every $f\in H$, $\widetilde f\equiv \mu(\CK^{\spe})f \mod [H, H]$. 
\end{lemma}

\begin{proof}
Let $\CK\subset \CK^{\spe}$ be a compact open subgroup such that $f$ is $\CK$-biinvariant. Then $\widetilde f(g)=\sum_{\bar k\in \CK^{\spe}/\CK} f(\bar k \i g \bar k) \mu(K)$. For $x\in G$, denote by ${}^xf$ the function ${}^xf(g)=f(x  \i gx)$. Then $$\widetilde f=\mu(\CK)\sum_{\bar k\in \CK^{\spe}/\CK} {}^{\bar k}f\text{ in } H.$$
It is well known that in the cocenter $\bar H$, $\bar f=\overline {{}^{x}f}$. The claim follows.
\end{proof}

\subsection{} Now we define $r_P(f) \in \bar H(M)$. We set 
$r_P(f)$ to be the unique element in $\bar H(M) \otimes \BQ$ such that \begin{equation}\label{C-r} \mu_P(\CK^{\spe}\cap P) \, r_P(f) \text{ equals  the image of } {\widetilde f}\,^{(P)} \text{ in } \bar H(M).\end{equation} 

Note that $\BZ[\frac{1}{p}]$ is an integral domain. Thus the element $r_P(f)$, as an element in $\bar H(M) \otimes \BQ$, is uniquely determined by (\ref{C-r}). In the case where $\mu_P(\CK^\spe\cap P)$ is invertible in $R$, one may also regard $r_P(f)$ as an function in $H(M)_R$. However, in general, $r_P(f)$ is not represented by a function in $H(M)$ in a natural way. But as we will see now, $r_P$ is well defined at the level of the cocenter, i.e., $r_P(\bar H)\subseteq \bar H(M)$.

\begin{lemma}
The restriction map $r_P$ preserves the $\mathbb Z[\frac 1p]$-structure, i.e., $r_P(\bar H)\subseteq \bar H(M).$ \end{lemma}

\begin{proof}
By Lemma \ref{l:twist}, $\widetilde f=\mu(\CK^\spe) f$ in $\bar H$, hence given (\ref{e:rP-explicit}), it is sufficient to show that $\frac{\mu(\CK^\spe)}{\mu_M(\CK^{\spe}\cap M)\mu_N({\CK^{\spe}\cap N})}\in \BZ[\frac 1p]$. Clearly, $\mu_N({\CK^{\spe}\cap N})^{-1}\in \BZ[\frac 1p]$. Moreover, with our normalization of measures $\frac{\mu(\CK^\spe)}{\mu_M(\CK^{\spe}\cap M)}=\frac{\#\bar {\CK}^\spe}{\#\bar M}$, where $\bar {\CK}^\spe$ and $\bar M$ are the reductive quotients of $\CK^\spe$ and $\CK^{\spe}\cap M$, respectively. These are finite reductive groups and $\bar M$ is a Levi subgroup of $\bar {\CK}^\spe$, hence the ratio is an integer. 
\end{proof}

%Moreover, as $\mu(\CK^\spe)$ is divisible by $C$, by Lemma \ref{l:twist}, $r_P(f) \in \bar H(M)$. \remind{add more explanation, c.f. Lemma 2.3 and Cor 2.5}

The normalization scalar $\mu_P(\CK^{\spe}\cap P)$ in the definition is needed in order for $r_P(f)$ to be in duality with parabolic induction. We explain this in an example. Notice first that if $\CK$ is any compact open subgroup of $G$ and $\pi'$ is a representation of $G$, then $\Tr(\delta_\CK,\pi')=\mu(\CK) (\pi')^{\CK}.$ Then for any $\pi \in \fkR_R(M)$, we have $$\Tr(\delta_{\CK^\spe}, i_M^G(\pi))=\mu(\CK^\spe)\dim(i_M^G(\pi))^{\CK^\spe}=\mu(\CK^\spe)\dim\pi^{\CK^\spe\cap M}.$$
We also have that $\widetilde \delta_{\CK^\spe}=\mu(\CK^\spe) \delta_{\CK^\spe}$ and then $\widetilde\delta_{\CK^\spe}^{(P)}=\mu(\CK^\spe) \mu_N(\CK^\spe\cap N) \delta_{\CK^\spe\cap M}$, where we used that $mn\in \CK^\spe$ if and only if $m\in \CK^\spe\cap M$ and $n\in \CK^\spe\cap N$ and that $\delta_P=1$ on $\CK^\spe\cap M$. Therefore $$\Tr({\widetilde \d_{\CK^\spe}}^{(P)}, \pi)=\mu(\CK^\spe) \mu_N(\CK^\spe\cap N)\mu_M(\CK^\spe\cap M) \dim\pi^{\CK^\spe\cap M}$$ and so \[\Tr(\delta_{\CK^\spe}, i_M^G(\pi))=\Tr(r_P(\delta_{\CK^\spe}),\pi).\]

\begin{theorem}\label{vd}
The map $H\to \bar H(M)$, $f \mapsto r_P(f)$ induces a map $\bar r_M: \bar H \to \bar H(M)$. 
\end{theorem}

\begin{proof}
It is sufficient to show that $f\in H$ and ${}^xf$, $x\in G$, have the same image in $\bar H(M)$. It is clear that $\widetilde{f}=\widetilde{{}^k f}$, for all $k\in \CK^{\spe}$, and therefore, by the Iwasawa decomposition, it is sufficient to consider the case $x=p_0\in P$.

Note that for any $p \in P$ and $f \in H$, we have that $({}^p f)^{(P)}={}^{\pi_P(p)} (f^{(P)})$ and thus $$({}^p f)^{(P)} \equiv f^{(P)} \mod [H(M), H(M)].$$ 

Now we have that $$\mu_P(\CK^\spe\cap P)~\widetilde{{}^{p_0} f}^{(P)}= \int_{\CK^{\spe}} ({}^{k p_0}f)^{(P)} dk= \int_{\CK^{\spe}} ({}^{p_{0, k} k_1}f)^{(P)} dk \equiv  \int_{\CK^{\spe}} ({}^{k_1}f)^{(P)} dk.$$ The claim then follows since $\mu$ is an invariant measure and therefore the map $$\pi_2: G \to \CK^{\spe},\quad  p k' \mapsto k'$$ gives a homeomorphism from $\CK^{\spe} p_0$ onto $\CK^{\spe}$ which preserves the Haar measure. 
\end{proof}

\subsection{} We would like to compare the extension of $\bar r_M$ over $R$ with the categorical restriction functor $\bar r^{\cat}_{M,  R}$. To this end, we recall certain basic facts about the Jacquet functor. 

For every open compact subgroup $\mathcal K$ in a group $L$ (which could be $G$, $P$, or $M$), such that $\mu_L(\CK)$ is invertible in $R$, denote by $e_{\mathcal K x\mathcal K}=\mu_L(\mathcal K)^{-1}\delta_{\CK x \CK}$ and $e_{x\mathcal K}=\mu(\mathcal K)^{-1}\delta_{x \CK}$. These are elements of $H(L)_R$. When $\CK$ is a subgroup of $G$ (and similarly for $M$) we will implicitly choose $\CK$ sufficiently small (e.g., a subgroup of the pro-p Iwahori subgroup) so that $\mu_G(\CK)$ is invertible in $R$.

Now suppose $\CK$ is a compact open subgroup of $P$ with a decomposition $\CK=(\CK\cap M)(\CK\cap N)$. Then it is easy to check that $\delta_{p\CK}^{(P)}=\mu_N(\CK\cap N) \delta_P^{\frac 12}\delta_{m(\CK\cap M)}$ and therefore 
\begin{equation}\label{P-proj}
e_{p\CK}^{(P)}=\delta_P^{\frac 12} e_{\pi_P(p\CK)}.
\end{equation}

The map $\pi_P: P\to M$ makes $H(M)_R$ into an $H(P)_R$-bimodule. More precisely, for every $h\in H(P)$, define $\pi_P(h)\in H(M)$ by extending linearly the definition $\pi_P(e_{p\CK})=e_{\pi_P(\CK)}$, see (\ref{P-proj}). Then the right action of $H(P)$ on $H(M)$ is given by
\begin{equation}\label{PM-action}
f_M\cdot h=f_M\star \pi_P(h),\quad \text{for all } h\in H(P)_R, \ f_M\in H(M)_R,
\end{equation}
in the right hand side, the convolution being in $H(M)_R.$ 

Every smooth $G$-representation $V$ gives rise to a smooth $P$-representation by restriction. Using the equivalence of categories, this implies that every $H_R$-module can be viewed as an $H(P)_R$-module. Precisely, if $v\in V$ is fixed by an open compact subgroup $\CK$ of $G$, and $\CK'$ is a compact open subgroup of $P$ such that $\CK'\subset \CK$, then $$\pi(e_{p\CK'})v=\pi(e_{p\CK})v,\  p\in P.$$ 
We will need the following result.
\begin{proposition}[{\cite[III.2.10]{Re}}]\label{p:renard} For every $V\in \mathfrak R_R(G)$, there is a natural isomorphism
\[r_{M,R}(V)\cong \delta_P^{\frac 12} H(M)_R\otimes_{H(P)_R}V.
\]
The isomorphism is induced by the map $v\mapsto e_{\CK}^{(P)}\otimes v,$ $v\in V$, where $\CK$ is a compact open subgroup of $P$ such that $\CK\cdot v=v.$
\end{proposition}

\begin{comment}
Let $\CK$ be an open compact subgroup of $G$. We say that $\CK$ is {\it of Bernstein-Deligne type} (in short, BD-type) if

\begin{itemize}
\item $\CK$ is totally decomposed with respect to the minimal parabolic $P_0$ in the sense of \cite[\S 1.1]{Bu};

\item $\CK$ is a normal subgroup of a maximal special subgroup of $G$. 
\end{itemize}

It is easy to see that congruence subgroups of a maximal special subgroup are of BD-type.
\end{comment}

\begin{theorem}\label{cat-r}
The restriction functor $\bar r^{\cat}_M$ equals the extension of $\bar r_M$ over $R$. 
\end{theorem}

\begin{proof}
 As it is well known, the Jacquet functor maps finitely generated smooth $G$-modules to finitely generated smooth $M$-modules. The algebra $H$ itself is not finitely generated as a left $H$-module (as it is not unital), so instead we need to work with the finitely generated projective modules $C_c(G/\CK)_R=H\otimes_{H(\CK)} 1_{\CK,R}$, where $\CK$ are compact open subgroups of $G$. Of course, $C_c(G/\CK)$ is the space of right $\CK$-invariant compactly supported $R$-functions. Set 
\[V_\CK=\delta_P^{\frac 12} H(M)\otimes_{H(P)} C_c(G/\CK)_R.
\]
If $f\in H$ is a locally constant function, the image $\bar f\in \bar H$ is represented by the pair $\varinjlim_{\CK}[C_c(G/\CK)_R,m_f]$, where $\CK$ is a compact open subgroup of $G$ such that $f$ is $\CK$-biinvariant, and $m_f: C_c(G/\CK)_R\to C_c(G/\CK)_R$ is right convolution by $f$. We therefore need to compute the image of $\varinjlim_{\CK}\tr_{H(M)}(V_\CK,m_f)$ in the categorical description of the restriction $\bar r_M^{\cat}$. Computing the trace is compatible with this direct system and it stabilizes, i.e., it does not change for sufficiently small open compact subgroups $\CK$.

Let $\CK_1$ be a compact open subgroup of $G$, and let $f$ be a right $\CK_1$-invariant $R$-function. Without loss of generality, we may assume that $f=\delta_{gK_1}$ for a fixed $g\in G$. Let $\CK$ be a compact open subgroup of $\CI'$ such that $\CK\subset \CK_1$ and $\CK$ is normal in $\CK^\spe$.  Using the Cartan decomposition $G=P\CK^\spe$, we may write $G=\sqcup_{i} Pg_i\CK,$ where $\{g_i\}$ are a (finite) set of representatives of $\CK^\spe\cap P\backslash \CK^\spe/K.$ Notice that, by our assumption on $\CK$, $g_i\CK g_i^{-1}=\CK$, a fact that we will use repeatedly below. The space $C_c(G/\CK)_R$ is a free left $H(P)$-module of finite rank:
\[C_c(G/\CK)_R=\bigoplus_i H(P)\cdot \delta_{g_i\CK}.
\]
Here $H(P)$ acts of $C_c(G/\CK)$ as explained in the paragraph before Proposition \ref{p:renard}. Concretely, 
\[\delta_{pg_i\CK}=\frac 1{\mu_P(P\cap\CK)}\delta_{(P\cap ^{p}\CK)p}\star \delta_{g_i\CK},\quad p\in P.
\] 
To compute the categorical restriction, we therefore have:
\[\bar r^{\cat}_M(f)=\tr_{H(M)}(V_\CK,m_f)=\delta_P^{\frac 12}\pi_P\left(\tr_{H(P)}(C_c(G/\CK)_R,m_f)\right),
\]
where for the second equality, we used (\ref{P-hom}), translated into the equivalent setting of $H(P)$-modules. We have $\delta_{g_i\CK}\star f=\delta_{g_i\CK}\star \delta_{g\CK_1}=\delta_{g_i\CK g \CK_1 g_i^{-1}}\star \delta_{g_i\CK}.$ The part of $\delta_{g_i\CK}\star f$ that contributes to $\tr_{H(P)}$ comes from
\begin{align*}g_i\CK g \CK_1 g_i^{-1}\cap Pg_i\CK&\cong (g_i\CK g \CK_1 g_i^{-1}\cap P)\cdot (g_i\CK g_i^{-1})=(g_i\CK g \CK_1 g_i^{-1}\cap P)\cdot \CK\\&\cong (g_i\CK g \CK_1 g_i^{-1}\cap P)\times_{\CK\cap P}\CK.
\end{align*}
Thus
\begin{equation*}
\tr_{H(P)}(C_c(G/\CK)_R,m_f)=\sum_i \frac{\mu_G(\CK)}{\mu_P(\CK\cap P)} \delta_{g_i\CK g \CK_1 g_i^{-1}\cap P}.
\end{equation*}
For every $i$, in the cocenter, we have $\delta_{g_i\CK g \CK_1 g_i^{-1}\cap P}\equiv \delta_{g'\CK g \CK_1 (g')^{-1}\cap P}$, for all $g'\in (\CK^\spe\cap P)g_i\CK.$ Moreover, $\mu_G((\CK^\spe\cap P) g_i\CK)=\frac{\mu_G(\CK)}{\mu_P(\CK\cap P)}\mu_P(\CK^\spe\cap P).$ This means that:

\[ \mu_P(\CK^\spe\cap P)\tr_{H(P)}(C_c(G/\CK),m_f)=\int_{\CK^\spe} \delta_{g'\CK g \CK_1 (g')^{-1}\cap P}~dg'.
\]
On the other hand, notice that
\[\widetilde f\mid_P=\int_{\CK^\spe}\delta_{g'g\CK_1(g')^{-1}\cap P}~dg'.
\]
By taking $\CK$ sufficiently small, we may assume that $\CK\subset g\CK_1 g^{-1}$, and therefore:
\[\mu_P(\CK^\spe\cap P)\tr_{H(P)}(C_c(G/\CK),m_f)=\widetilde f\mid_P.
\]
$\mu_P(\CK^\spe\cap P)\bar r^{\cat}_M(f)=\delta_P^{\frac 12}\pi_P(\widetilde f\mid_P)=\widetilde f^{(P)}.$ Hence, on the level of the cocenter, $\bar r^{\cat}_M=\bar r_{M}$.
\end{proof}

\subsection{} We need to investigate the relation between $\bar r_M$ and twists by $w\in W.$ The following lemma will allow us to transfer known results from complex numbers to our setting.

\begin{lemma}\label{l:commute}
The natural map $\bar H\to \bar H_\BC$ is injective.
\end{lemma}

\begin{proof}
By \cite[\S 2a]{Vig2} (``changement de base''), the natural homomorphism $\bar H\otimes_{\BZ[\frac 1p]} \BC\to \bar H_\BC$ is an isomorphism. On the other hand $\BC$ is flat over $\BZ[\frac 1p]$, hence the claim follows.
\begin{comment}
We use the fact that the commutator subspace is spanned by differences of the form $\delta_{X}-{}^g\delta_X$, where $X$ is a compact set. It is obvious that $[H,H]\subseteq [H_\BC,H_\BC]\cap H$. For the converse, let $f\in [H_\BC,H_\BC]\cap H$ be given. Then we write $f=\sum_{i=1}^n a_i \psi_i$, where $\psi_i=\delta_{X_i}-{}^g\delta_{X_i}$ and $a_i\in\BC$. All we will use is that $\psi_i$ takes integer values, more precisely only the values $\{-1,0,1\}$. Choose a basis $\mathcal B$ of $\BC$ as a $\mathbb Q$-vector space, such that $1\in \mathcal B$. There exists a finite subset $I\subset \mathcal B$ such that all $a_i$ can be expressed in terms of $I$, $a_i=\sum_{\alpha\in I}q_{i,\alpha}\alpha$, where $q_{i,\alpha}\in \BQ$. We may assume $1\in I$. This means that 
\[f=\sum_{i}q_{1,\alpha}\psi_i+\sum_{\alpha\in I\setminus \{1\}} (\sum_{i} q_{i,\alpha}\psi_i)\alpha,
\]
and therefore $\sum_{i} q_{i,\alpha}\psi_i=0$ for all $\alpha\neq 1$, and $f=\sum_{i}q_{1,\alpha}\psi_i$. In other words, $f\in [H_\BQ,H_\BQ]$.
\end{comment}
\end{proof}

%Let $G_\rs$ denote the set of regular semisimple elements in $G$. 
%Recall the statement of the Spectral Density Theorem:

%\begin{enumerate}
%\item[(SDT)] If $\Tr_R(f,\pi)=0$ for all irreducible smooth $R$-representations $\pi$, then $f\in [H_R,H_R]$. 
%\item[(GDT)] If $\mathcal O^G_g(f)=0$ for all $g\in G_\rs$, then $f\in [H_R,H_R]$. Here $\mathcal O^G_g(f)=\int_{G/T} f(xgx^{-1}) d\bar x$ is the semisimple orbital integral, where $T=Z_G(g)$ is a Cartan subgroup.
%\end{enumerate}

Now we recall the Spectral Density Theorem. 

\begin{theorem}\label{SDT}
Let $f \in H_{\BC}$. If $Tr_{\BC}(f, \pi)=0$ for all irreducible smooth $\BC$-representations $\pi$, then $f \in [H_\BC, H_\BC]$. 
\end{theorem}

The Spectral Density Theorem was first proved by Kazhdan \cite{Kaz,Ka2} when $\textup{char} F=0$ or $G$ is split. The general case was obtained recently by Henniart and Lemaire in \cite{HL}. Note that for algebraically closed field $R$ of positive characteristic, the analogous statement (by replacing $\BC$ by $R$) may fail. 

%By \cite{HL}, (SDT) is known when $R=\BC$. (Previously, (SDT) had been known when $R=\BC$ and $\textup{char} F=0$ or $G$ is split by the results of Kazhdan \cite{Kaz,Ka2}.)

\smallskip

Now we prove the following statement on the relation between $\bar r_M$ and twisted by $w \in W$. 

\begin{proposition}\label{r-w}
Let $M, M'$ be standard Levi subgroups and $w \in W_G$ with $M'={}^{\dot w} M$. Then 
\begin{equation}\label{e:twist-r}
\bar r_{M'}={\dot w} \circ \bar r_M: \bar H \to \bar H.
\end{equation}
\end{proposition}

\begin{proof}
By the adjunction property (\ref{e:dual}) and by the compatibility of restrictions, Theorem \ref{cat-r}, we see that
\begin{equation}
\Tr_\BC^G(h_G,i_{M,\BC}(\sigma))=\Tr_\BC^M(\bar r_{M,\BC}(h_G),\sigma),\quad h_G\in \bar H_\BC.
\end{equation}
On the other hand, it is well known that 
\begin{equation}i_{M,\BC}\circ {\dot w}^{-1}=i_{M',\BC},
\end{equation}
see for example \cite{BDK}. Therefore by Theorem \ref{SDT}, $\bar r_{M',\BC}={\dot w} \circ \bar r_{M,\BC}$. The claim follows then from Lemma \ref{l:commute}. Notice that one only needs to use spectral density theorem for the proper Levi subgroups and for $R=\BC$.
\end{proof}

\section{Some compatibility results}

\subsection{} We first discuss the compatibility between the Newton decomposition and the induction maps on the cocenter. 

We recall the explicit formula in \cite{hecke-2}. Let $M$ be a standard Levi subgroup and $v \in V_+$ with $M=M_v$. Set $P=P_v$. By \cite[Theorem A]{hecke-2}, the map $$\d_{m \CK_M} \mapsto \d_P(m)^{-\frac{1}{2}} \frac{\mu_M(\CK_M)}{\mu_G(\CK_M \CI_n)} \d_{m \CK_M \CI_n}+[H, H] \text{ for } n \gg 0$$ gives a well-defined surjection $$\bar i_v: \bar H(M; v) \to \bar H(v).$$ Here $\bar H(M; v) \subset \bar H(M)$ is the Newton component corresponding to $v$. By \cite[Theorem B]{hecke-2}, the extension of $\bar i_v$ over $R$ is adjoint to the Jacquet functor $r_{M, R}: \fkR_R(G) \to \fkR_R(M)$. Therefore, the extension of $\bar i_v$ over $R$ equals $\bar i^{\cat}_{M, R}$ on the reduced cocenter. In particular, we have 

\begin{proposition}
Let $v \in V_+$ and $M=M_v$. Then $\bar i^{\cat}_{M, R}$ maps $\bar H^{\red}_R(M; v)$ onto $\bar H^{\red}_R(v)$. 
\end{proposition}

Note that the maps $\bar i_*$ are defined only for the Newton strata $\bar H(M; v)$ of $\bar H(M)$ with $M=M_v$. It is a challenging problem to give an explicit formula for $\bar i_M: \bar H(M) \to \bar H$. 

\subsection{} Next we show that the Newton decomposition is compatible with the restriction maps on the cocenter. 

Let $V_+^M$ be the set of $M$-dominant elements in $V$ and $\Omega_M=\tW(M)/W_a(M)$, where $\tW(M)$ is the Iwahori-Weyl group of $M$ and $W_a(M)$ is the associated affine Weyl group. Set $\aleph_M=\Omega_M \times V^M_+$.  As explained in \cite[\S 1.5]{hecke-2}, there is a natural map $$\aleph_M \to \aleph, \qquad \nu \mapsto \bar \nu.$$ On the $V$-factor, this maps sends any $M$-dominant element in $V$ to the unique $G$-dominant element in its $W_G$-orbit. 

\begin{proposition}\label{r-N}
Let $\nu \in \aleph$ and $M$ be a standard Levi subgroup. Then $$\bar r_M(\bar H(\nu)) \subseteq \oplus_{\nu' \in \aleph_M; \bar \nu'=\nu} \bar H(M; \nu').$$
\end{proposition}

\begin{proof}
Let $f=\d_X$, where $X \subset G(\nu)$. By definition, the support of $\widetilde f^{(P)}$ equals $\cup_{k \in \CK^{\spe}} \pi_P (k X k \i \cap P)$. We have $k \i X k \cap P \subset G(\nu) \cap P$. Now the statement follows from the following Lemma. 
\end{proof}

\begin{lemma}
Let $P=M N$ be a standard parabolic subgroup and $\nu \in \aleph$. If $p \in P \cap G(\nu)$, then $\pi_P(p) \in M(\nu')$ for some $\nu' \in \aleph_M$ with $\bar \nu'=\nu$. 
\end{lemma}
\begin{remark}
The original proof was more complicated. The following simplification was suggested by S. Nie. 
\end{remark}

\begin{proof}
Let $p=m u$ with $m \in M$ and $u \in N$. We assume that $m \in M(\nu')$ for some $\nu' \in \aleph_M$. 

Let $\l \in X_*(Z)$ such that $\<\l, \a\>=0$ for any relative root $\a$ in $M$ and $\<\l, \a\>>0$ for any relative root in $N$. By \cite[Proposition 4.2]{hecke-2}, $m \in G(\bar \nu')$. By \cite[Theorem 3.2]{hecke-1}, there exists $n \in \BN$ such that $m \CI_n \subset G(\bar \nu')$. By our assumption on $\l$, there exists $l \in \BN$ such that $t^{l \l} u t^{-l \l} \in \CI_n \cap N$. Hence $$t^{l \l} p t^{-l \l}=m t^{l \l} u t^{-l \l} \in m \CI_n \subset G(\bar \nu').$$ However, $t^{l \l} p t^{-l \l}$ is conjugate to $p$ and thus is contained in $G(\nu)$. By \cite[Theorem 2.1]{hecke-1}, we must have $\nu=\bar \nu'$. 
\end{proof}

\section{The main result}

We first prove the following weak form of the Mackey formula for the cocenter. 

\begin{proposition}\label{mackey}
Let $M$ be a standard Levi subgroup and $v \in V_+$ with $M=M_v$. Then for any $f \in \bar H(M; v)$, we have $$\bar r_{M} \circ \bar i_v(f) \in f+\sum_{w \in {}^{M} W^M; w \neq 1} \bar H(M; w(v)),$$ where ${}^M W^M$ is the subset of $W_G$ consisting of elements of minimal length in their $W_M \times W_M$-cosets.
\end{proposition}

\begin{remark}\label{remark-mackey}
We expect that $$\bar r_{M} \circ \bar i_v(f)=\sum_{w \in {}^{M} W^M} \bar i^{M}_{w(v)} \circ w \circ \bar r^M_{M \cap {}^{w \i} M} (f).$$

In the reduced cocenter $\bar H^{\red}_R$, this can be obtained via the adjunction formula and the Mackey formula in $\fkR_R(G)$ \cite[\S I.5.5]{Vig}. The equality in the reduced cocenter is enough for the application to the trace Paley-Wiener Theorem, but for the application to the abstract Selberg principle, we need an equality in the cocenter. 
\end{remark}

\begin{proof}
It suffices to consider the case where $f=\d_{m \CK_M}$, where $m \in M$ and $\CK_M$ is an open compact subgroup of $M$ such that $m \CK_M \in M(v)$. By definition, $\mu_P(\CK^{\spe} \cap P) \bar r_M \circ \bar i_v(f)$ is represented by $$\frac{\mu_M(\CK_M)}{\mu_G(\CK_M \CK')} \int_{\CK^{\spe}} \d^{(P)}_{k X k \i \cap P} dk,$$ where $X=m \CK_M \CK'$ for some sufficiently small open compact subgroup $\CK'=\CK'_N \CK'_M \CK'_{N^-}$ of $G$. 

By \cite[Proposition 2.3]{hecke-2}, any element in $m \CK_M \CK'$ is conjugate by an element in $\CK^{\spe}$ to an element in $m \CK_M$. 

Let $w \in {}^M W^M$. We first show that \[\tag{a} \pi_P(k m \CK_M k \i \cap P) \subseteq M(w(v)) \quad \text{ for any } k \in P \dot w P.\] Note that $\pi_P(p X' p \i \cap P)=\pi_P(p) \pi_P(X' \cap P) \pi_P(p) \i$ for any $p \in P$ and $X' \subset G$. It suffices to prove (a) for $k \in \dot w P$. 

Let $k=\dot w p$ for $p \in P$. Note that $k m \CK_M k \i \cap P \subseteq \dot w P \dot w \i \cap P$. By \cite[Theorem 2.8.7]{Ca}, we have $$\dot w P \dot w \i \cap P \cong M'  \times (\dot w N \dot w \i \cap M) \times (\dot w M \dot w \i \cap N) \times (\dot w N \dot w \i \cap N),$$ where $M'=\dot w M \dot w \i \cap M$ is a standard Levi subgroup of $G$. 
Then $$\pi_P(k m \CK_M k \i \cap P) \subseteq \pi_P(\dot w \pi_P(p m \CK_M p \i) \dot w \i \cap P) (\dot w N \dot w \i \cap M).$$ 

Note that $\pi_P(\dot w \pi_P(p m \CK_M p \i) \dot w \i \cap P) \subseteq M'$. We have \begin{align*} \pi_P(\dot w \pi_P(p m \CK_M p \i) \dot w \i \cap P) &=\pi_{M'}(\dot w \pi_P(p m \CK_M p \i) \dot w \i \cap P) \subseteq M'(w(v)) \\ & \subseteq M(w(v)).\end{align*} Here the first inequality follows from Proposition \ref{r-N} and the second inequality follows from the fact that $w(v)$ is $M$-dominant. 

Hence we have $\pi_P(k m \CK_M k \i \cap P) \subseteq M'(w(v)) (\dot w N \dot w \i \cap M)$. As $\<w(v), \a\>>0$ for any root $\a$ in $\dot w N \dot w \i \cap M$, by the proof of \cite[Proposition 2.3 (1)]{hecke-2}, any element in $M'(w(v)) (\dot w N \dot w \i \cap M)$ is conjugate in $M$ to an element in $M'(w(v))$. Therefore $\pi_P(k m \CK_M k \i \cap P) \subseteq M(w(v))$ and (a) is proved. 

Now we compute the Newton component $f'$ of $\mu_P(\CK^{\spe} \cap P) \bar r_{M} \circ \bar i_v(f)$ with Newton point $v$. Note that $\CK^{\spe}=\sqcup_{w \in {}^M W^M} \CK^{\spe} \cap P \dot w P$. By (a), $$f'=\frac{\mu_M(\CK_M)}{\mu_G(\CK_M \CK')} \int_Y \d^{(P)}_{k X k \i \cap P} dk,$$ where $Y$ is the subset of $\CK^{\spe}$ consisting of elements of the form $p k'$ with $p \in \CK^{\spe} \cap P$ and $k' \in \CK^{\spe}$ such that $k' X (k') \i \cap m \CK_M \neq \emptyset$. As $\CK'_{N}$ is sufficiently small, for any $u \in \CK^{\spe} \cap N^-$, we have that $u X u \i \subseteq (m \CK_M \CK'_N) N^-$ and if $u X u \i \cap m \CK_M \CK'_N \neq \emptyset$, then $u X u \i \cap m \CK_M \CK'_N=m \CK_M \CK'_N$. By the proof of \cite[Proposition 2.3 (1)]{hecke-2}, $Y=(\CK^{\spe} \cap P) \times Y'$, where $Y'=\{u \in \CK^{\spe} \cap N^-; m \CK_M \CK'_N \subseteq u X u \i\}$ and $\mu_{N^-} (Y')=\mu_{N^-}(\CK'_{N^-})$. 

Note that for any $p=m' u \in P$ with $m_1 \in M$ and $u \in N$, we have that $\pi_P(\d_{p m \CK_M \CK'_N})=\mu_{N}(\CK'_N) \d_{m' m \CK_M (m') \i}$. Thus we have in $\bar H(M)$, \begin{align*} f' &=\frac{\mu_M(\CK_M)}{\mu_G(\CK_M \CK')} \int_Y \d^{(P)}_{k X k \i \cap P} dk \\ &=\frac{\mu_M(\CK_M)}{\mu_G(\CK_M \CK')} \mu_{N^-}(\CK'_{N^-}) \int_{\CK^{\spe} \cap P} \d^{(P)}_{k m \CK_M \CK'_N k \i} dk \\ &=\frac{\mu_M(\CK_M)}{\mu_G(\CK_M \CK')} \mu_{N^-}(\CK'_{N^-}) \mu_N(\CK^{\spe} \cap N) \mu_{N}(\CK'_N) \int_{\CK^{\spe} \cap M} \d_{m' m \CK_M (m') \i} d m' \\ &=\frac{\mu_M(\CK_M)}{\mu_G(\CK_M \CK')} \mu_{N^-}(\CK'_{N^-}) \mu_N(\CK^{\spe} \cap N) \mu_{N}(\CK'_N) \mu_M({\CK^{\spe} \cap M}) f \\ &=\mu_P(\CK^{\spe} \cap P) f \in \bar H(M).
\end{align*}
Here the fourth equality follows from that $\d_{m' m \CK_M (m') \i}$ and $f=\d_{m \CK_M}$ have the same image in $\bar H(M)$ and the last equality follows from the fact that $\mu_G(\CK_M \CK')=\mu_M(\CK_M) \mu_N (\CK'_N) \mu_{N^-}(\CK'_{N^-})$ and $\mu_P(\CK^{\spe} \cap P)=\mu_M(\CK^{\spe} \cap M) \mu_N(\CK^{\spe} \cap N)$.  
\end{proof}

Combining the weak form the Mackey formula with the Newton decomposition, we have the following result. 

\begin{theorem}\label{main'}
Let $v \in V_+$ and $M=M_v$. Then the restriction maps $\bar r_M: \bar H(G; v) \to \bar H(M)$ and $\bar r_M: \bar H_R^{\red}(G; v) \to \bar H_R^{\red}(M)$ are injective. 
\end{theorem}

\begin{proof}
Let $f \in \bar H(v)$. By \cite[Theorem A]{hecke-2}, we have $f=\bar i_M(f')$ for some $f' \in \bar H(M; v)$. As $W_M$ is the centralizer of $\nu$ in $W$, for $w \in {}^M W^M$, we have $w (v)=v$ if and only if $w=1$. %By Theorem \ref{mackey} and remark \ref{remark-mackey} (2), we have $\bar r_M(f)=\bar r_M \bar i_M(f')=\sum_{w \in {}^M W^M} \bar i^M_{M \cap {}^w M} \circ w \circ \bar r^M_{M \cap {}^{w \i} M}(f')$ and $\bar i^M_{M \cap {}^w M} \circ w \circ \bar r^M_{M \cap {}^{w \i} M}(f') \in \bar H^{\red}(M; w(\nu))$.

If $\bar r_M(f)=0 \in \bar H(M)$, then by the Newton decomposition on $\bar H(M)$ (see \cite[Theorem B]{hecke-1}) and Proposition \ref{mackey}, we have $f'=0 \in \bar H(M)$. 

If $\bar r_M(f) \in \ker \Tr^M_R$, then by the compatibility between the Newton decomposition and the trace map (see \cite[Theorem 6.3]{hecke-2}), we have $f' \in \ker \Tr^M_R$ and thus $f \in \ker \Tr^G_R$. 
\end{proof}

\subsection{} We define the elliptic part of the cocenter by 
\begin{gather*}
\bar H^{\el}=\{f \in \bar H; \bar r_M(f)=0 \text{ for any proper standard Levi } M\}, \\
\bar H^{\el, \red}_R=\{f \in \bar H^{\red}; \bar r^{\red}_{M, R}(f)=0 \text{ for any proper standard Levi } M\}.
\end{gather*}

Now we compare the elliptic part and rigid part of the cocenter. 

\begin{theorem}\label{main}
We have $$\bar H^{\el} \subset \bar H^\rig, \quad \bar H^{\el, \red}_R \subset \bar H^{\rig, \red}_R.$$
\end{theorem}

\begin{proof}
We prove the first inclusion. The second one is proved in the same way. 

Let $f \in \bar H$. We may write $f$ as $f=\sum_{v \in V_+} f_{v}$, where $f_{v} \in \bar H(v)$ is the corresponding Newton component of $f$. By Proposition \ref{r-N}, $$\bar r_{M}(f_v) \in \oplus_{v' \in W_G(v); v'  \text{ is $M$-dominant}} \bar H(M; v').$$ Note that $W_G(v_1) \cap W_G(v_2)=\emptyset$ for distinct $v_1, v_2 \in V_+$. By the Newton decomposition on $\bar H(M)$ (see \cite[Theorem B]{hecke-1}), we have that $\bar r_M(f_v)=0$ for all proper Levi subgroups $M$. By Theorem \ref{main'}, $f_{v}=0$ for all $v \in V_+$ with $M_v \neq G$. Therefore $f \in \bar H^{\rig}$. 
\end{proof}

\section{Application: the trace Paley-Wiener Theorem}

\subsection{}\label{good} Let $M$ be a standard Levi subgroup. Denote by $M^0$ the subgroup of $M$ generated by all the compact subgroups of $M$. Then $M^0$ is open in $M$ and $M/M^0$ is an abelian group of finite rank \cite[I.1.4]{Vig}. An {\it unramified character} of $M$ over $R$ is a group homomorphism $G\to R^\times$ which is trivial on $M^0$. We denote by $\Psi(M)_R$ the set of all the unramified characters of $M$ over $R$. This is a group under multiplication. In fact, as $M/M^0\cong \mathbb Z^n$ for $n\ge 0$, we have $\Psi(M)_R\cong (R^\times)^n$. The group $\Psi(M)_R$ acts by multiplication on $\mathfrak R_R(M)$.

\smallskip

We recall the definition of good forms introduced in \cite{BDK}.  Let $f \in \fkR_R(G)^*$. We say that $f$ is {\it good} if: 

\begin{itemize}
\item There exists an open compact subgroup $\CK$ such that $f(\pi)=0$ if $\pi$ has no nonzero $\CK$-fixed points. 

\item For every standard Levi subgroup $M$ and $\s \in \fkR_R(M)$, the function $$\Psi(M)_R \to R, \quad \psi \mapsto f(i_M(\s \otimes \psi))$$ is a regular function. 
\end{itemize}

It is easy to see that $\Tr_R(f)$ is a good form for any $f \in \bar H$. 

In the rest of this section, we show that if $\sharp W_G$ is invertible in $R$, then the converse is also true, i.e., any good form comes from the cocenter of $H$. 

\begin{theorem}\label{trace-p-w}
Assume that $\sharp W_G$ is invertible in $R$. Then the trace map $\Tr_R: \bar H_R \to \fkR_R(G)^*_{\good}$ is surjective. 
\end{theorem}

\subsection{} We first recall the $A$-operator introduced in \cite[\S 5.5]{BDK}. 

For any standard Levi subgroup $M$, let $d(M)=\dim Z(M)$, where $Z(M)$ is the center of $M$. Let $N_M=\{w \in W_M \backslash W_G/W_M; w W_M w \i=W_M\}$. For any $l \in \BN$, we define $A_{l, R}=\Pi_{M; d(M)=l} (i_{M, R} \circ r_{M, R}-\sharp N_M)$. We define $$A_R=A_{d(G), R} A_{d(G)-1, R} \cdots A_{1, R}: \fkR_R(G) \to \fkR_R(G).$$

We have the Mackey formula \cite[\S I.5.5]{Vig} $$r_{M, R} \circ i_{M', R}=\sum_{w \in {}^M W^{M'}} i^M_{M \cap {}^{\dot w} M', R} \circ \dot w \circ r^{M'}_{M' \cap {}^{\dot w \i} M, R}: \fkR_R(M') \to \fkR_R(M).$$ 

Combining Proposition \ref{r-w} with Theorem \ref{cat-r}, we have

\begin{lemma}\label{con-w}
Let $M, M'$ be standard Levi subgroups of $G$ and $w \in W$ with ${}^{\dot w} M=M'$. Then $$i_{M, R}=i_{M', R} \circ \dot w: \fkR_R(M) \to \fkR_R(G).$$ 
\end{lemma}

\begin{remark}
This result is first proved in \cite[Lemma 5.4 (iii)]{BDK} for complex representations. The proof in {\it loc.cit.} is based on the Langlands classification. Notice that here, this follows from the cocenter results. Therefore Lema \ref{con-w} is valid for arbitrary algebraically closed field of characteristic not equal to $p$. 
\end{remark}

By the argument in \cite[\S 5.4 \& 5.5]{BDK}, we have

\begin{proposition}\label{A^2}
We have $A_R^2=a A_R$ for some positive integer $a$ whose prime factors divides $\sharp W_G$ and $$\ker A_R=\sum_{M \subsetneqq G} i_{M, R} (\fkR_R(M)).$$
\end{proposition}

Note that \cite{BDK} only consider the case where $R=\BC$. The essential ingredients used in {\it loc.cit.} are the Mackey formula and Lemma \ref{con-w}. As both the ingredients are known now for any algebraic closed field $R$ of characteristic not equal to $p$, Proposition \ref{A^2} is valid in this general situation as well. 

\subsection{}\label{bar-A-def} Let $\bar r_{M, R}^{\red}: \bar H^{\red}_R \to \bar H^{\red}_R(M)$ be the map induced from the map $\bar r_M$. Let $\bar i^{\red}_{M, R}: \bar H^{\red}_R(M) \to \bar H^{\red}_R$ be the map adjoint to the Jacquet functor $r_{M, R}$. Let $\bar A^{\red}_R: \bar H^{\red}_R \to \bar H^{\red}_R$ be the map adjoint to the $A$-operator on $\fkR_R(G)$, i.e. for $f \in \bar H^{\red}_R$ and $\pi \in \fkR_R(G)$, we have $$\Tr^G_R(f, A_R(\pi))=\Tr^G_R(\bar A^{\red}_R(f), \pi).$$

We also have the following description of the elliptic cocenter.

\begin{proposition}\label{im-A}
Assume that $\sharp W_G$ is invertible in $R$. Then $\bar H^{\el, \red}_R=\Im \bar A_R^{\red}$.
\end{proposition}

\begin{proof}
We have $A_R \circ i_{M, R}=0$ on $\fkR_R(M)$. Therefore $\bar r^{\cat}_M \circ \bar A_R^{\cat}=0$ on $\bar H_R^{\red}$. Hence $\Im \bar A_R^{\red} \subset \bar H^{\el, \red}$. On the other hand, let $f \in \bar H^{\el, \red}$, then by definition $\bar A_R^{\cat}(f)=\pm a (f)$ and thus $f \in \Im(\bar A_R^{\red})$. 
\end{proof}

\smallskip

\subsection{} 
For any open compact subgroup $\CK$ of $G$, let $H(G, \CK)$ be the space of compactly supported $\CK \times \CK$-biinvariant $\BZ[\frac{1}{p}]$-valued functions on $G$ and $H_R(G, \CK)=H(G, \CK) \otimes_{\BZ[\frac{1}{p}]} R$. Then we have $H=\varinjlim\limits_{\CK} H(G, \CK)$ and $H_R=\varinjlim\limits_{\CK} H_R(G, \CK)$. The trace map $\Tr_R: H_R \to \fkR(H_R)^*$ induces $$\Tr_R: \bar H_R(G, \CK) \to \fkR(H_R(G, \CK))^*.$$

Let $\CK$ be an open compact subgroup of $G$. For any $\nu \in \aleph$, let $H(G, \CK; \nu)=H(G, \CK) \cap H(\nu)$ and $\bar H(G, \CK; \nu)$ be its image in $\bar H$. Then $H(G, \CK) \supsetneqq \oplus_\nu H(G, \CK; \nu)$. However, by \cite[Theorem 4.1]{hecke-1} and \cite[Theorem 6.3]{hecke-2}, we have that

\begin{theorem}\label{In-dec}
Let $\CI_n$ be the $n$-th congruence subgroup of the Iwahori subgroup of $G$. Then $$\bar H(G, \CI_n)=\oplus_{\nu \in \aleph} \bar H(G, \CI_n; \nu), \qquad \bar H^{\red}_R(G, \CI_n)=\oplus_{v \in V_+} \bar H^{\red}_R(G, \CI_n; v).$$
\end{theorem}

\subsection{}\label{depth} We recall the relevant properties of the Moy-Prasad filtration. The original references are \cite{MP0,MP} where the notion of depth of a representation is defined for irreducible smooth complex representations. This is extended in \cite{Vig} to the case of mod-$l$ representations. Let $\bar F^{u}$ be a maximal unramified extension of $F$ and $\Gamma=\bar F^u/F$ be the Galois group. Let $\CB(\BG,\bar F^u)$ be the Bruhat-Tits building of $\BG/\bar F^u$ and let $\CB(\BG,F)=\CB(\BG,\bar F^u)^\Gamma$ be the building of $\BG/F$ as in \cite[\S3.1]{MP}. For every $x\in \CB(\BG,F)$, let $\BG_x$ be the parahoric subgroup defined by $x$, a subgroup of finite index in $\BG_x^\dagger=\{g\in \BG\mid g\cdot x=x\}$. For every $r\ge 0$, $\BG_{x,r}$ denotes the Moy-Prasad filtration subgroup of $\BG_x$. These subgroups are defined using the affine root subgroups $\{U_{\psi}\mid \psi(x)\ge r\}$, where $\psi$ are the affine roots of $\BG$ defined with respect to a maximal $\bar F^u$-split torus $T$. Each $\BG_{x,r}$ is normal in $\BG_x$, $\BG_{x,s}\subseteq \BG_{x,r}$ if $s\ge r$, and $[\BG_{x,r},\BG_{x,s}]\subseteq \BG_{x,r+s}.$ Define
\[\BG_{x,r^+}=\cup_{s>r} \BG_{x,s}.
\]
If $r>0$, then $\BG_{x,r}/\BG_{x,r^+}$ is abelian. Since $\BG_{x,r}$ is $\Gamma$-stable, set $G_{x,r}=G_x\cap \BG_{x,r}$ and $G_{x,r^+}=G_x\cap \BG_{x,r^+}.$ Then $G_{x,r}$ and $G_{x,r^+}$ are open normal subgroups of the parahoric subgroup $G_x$. 

\begin{comment}
An {\it unrefined minimal K-type} (\cite[\S3.4]{MP}) is a pair $(G_{x,r},\chi)$, where $x\in \CB(\BG,F)$, $r\ge 0$ such that $G_{x,r}\neq G_{x,r^+}$, and $\chi$ is a representation of $G_{x,r}/G_{x,r^+}$ such that
\begin{enumerate}
\item[(i)] If $r=0$, then $\chi$ is a cuspidal representation of the finite reductive group $G_{x,r}/G_{x,r^+}$ (the depth $0$ case);
\item[(ii)] If $r>0$, then $\chi$ is a nondegenerate character of the abelian group $G_{x,r}/G_{x,r^+}$ (the positive depth case).
\end{enumerate}
\end{comment}

Following \cite[Theorem 3.5]{MP} and \cite[II.5]{Vig} for every representation $(\pi,V)\in \textup{Irr}_R(G)$, define the {\it depth} of $\pi$ to be the rational number $r(\pi)\ge 0$ such that for some $x\in \CB(\BG,F)$, the space $V^{G_{x,r(\pi)^+}}$ of $G_{x,r^+}$-fixed vectors is nonzero and $r(\pi)$ is the smallest number with this property. The next result relates this notion to parabolic induction and restriction.

\begin{theorem}[{\cite[Theorems 4.5 and 5.2]{MP}, \cite[II.5.12]{Vig}}]\label{t:MP-ind} The functors of parabolic induction $i_{M,R}$ and parabolic restriction $r_{M,R}$ preserve the depth of representations.
\end{theorem}

As a corollary, we relate this to the case of representations with $\CI_n$-fixed vectors.

\begin{corollary} Let $M$ be a Levi subgroup of a parabolic subgroup $P$ and let $\CI_n$ be the $n$-th level Iwahori filtration subgroup of $G$.
\begin{enumerate}
\item The functor $i_{M,R}$ maps $\mathfrak R_R(H(M,\CI_n\cap M))$ to $\mathfrak R_R(H(G,\CI_n))$ and $r_{M,R}$ maps $\mathfrak R_R(H(G,\CI_n))$ to $\mathfrak R_R(H(M,\CI_n\cap M))$.
\item The reduced parabolic functor $i_{M,R}^\red$ maps $\bar H^{\red}_R(M, \CI_n\cap M)$ to $\bar H^{\red}_R(G, \CI_n)$ and $r_{M,R}^\red$ maps $\bar H^{\red}_R(G, \CI_n)$ to $\bar H^{\red}_R(M, \CI_n\cap M)$. 
\item In particular, $\bar A_R^\red$ maps $\bar H^{\red}_R(G, \CI_n)$ to itself.
\end{enumerate}
\end{corollary}

\begin{proof}
Claim (3) follows from (2) since $\bar A_R^\red$ is the linear combination of composition of maps of the form $i_{M,R}^\red\circ r_{M,R}^\red$. Claim (2) follows from (1) by the adjunction properties (\ref{e:dual}). It remains to justify Claim (1). Let $x_0\in  \CB(\BG,F)$ be such that $\CI=G_{x,0^+}$. We have $\CI_n=G_{x_0,(n-\epsilon)^+}$ for $\epsilon>0$ infinitesimally small. But then Claim (1) follows from Theorem \ref{t:MP-ind}, since for every $x\in \CB(\BG,F)$, there exists $g\in G$ such that $\CI_n\subseteq {}^gG_{x,(n-\epsilon)^+}$. (In other words, up to associates, $\CI_n$ is the smallest open compact subgroup of depth $(n-\epsilon)^+$.)
\end{proof}

\

Now we prove the elliptic trace Paley-Wiener theorem.

\begin{theorem}\label{disc}
Assume that $\sharp W_G$ is invertible in $R$. Let $n$ be a positive integer. Then the trace map $$\Tr_R: \bar H_R(G, \CI_n) \to A_R(\fkR(H_R(G, \CI_n)))^*_{\good}$$ is surjective.  

In particular,  the trace map $\Tr_R: \bar H_R \to A_R(\fkR_R(G))^*_{\good}$ is surjective. 
\end{theorem}

\begin{proof}
Using the action of the unramified central characters of $G$, we may reduce to the case where $G$ is semisimple. In this case, there are only finitely many $\nu \in \aleph$ with $M_\nu=G$ and $G(\nu)\neq \emptyset$. By \cite[Theorem 5.3]{hecke-1}, $\bar H_R(G, \CI_n; \nu)$ is finite dimensional for each $\nu$. Hence $\bar H_R(G, \CI_n)^\rig$ is finite dimensional. 

Let $f \in A_R(\fkR(H_R(G, \CI_n)))^*_{\good}$. Suppose that $f$ is not contained in the image of $\Tr_R$. Then there exists $$\pi \in A_R(\fkR(H_R(G, \CI_n))) \subset \fkR(H_R(G, \CI_n))$$ such that $\Tr_R(\bar H_R(G, \CI_n)^{\rig, \red}, \pi)=0$ and $f(\pi) \neq 0$. 

Let $f' \in \bar H^{\red}_R(G, \CI_n)$. Then $$\Tr_R(f', \pi)=\Tr_R(f', \frac{1}{a} A_R(\pi))=\frac{1}{a} \Tr_R(f', A_R(\pi))=\frac{1}{a} \Tr_R(\bar A^{\red}_R(f'), \pi).$$

By \S\ref{depth} and the Newton decomposition on $\bar H^{\red}_R(G, \CI_n)$ (see Theorem \ref{In-dec}), we have $$\bar A^{\red}_R(f') \in \bar H^{\red}_R(G, \CI_n) \cap \bar H^{\rig, \red}_R=\bar H_R(G, \CI_n)^{\rig, \red}.$$ Thus by the assumption, $\Tr_R(f', \pi)=0$ for any $f' \in \bar H_R(G, \CI_n)$. Therefore $\pi=0 \in 
\fkR(H_R(G, \CI_n))$ and $f(\pi)=0$. This is a contradiction. 

The ``in particular'' part follows from the fact that $H_R=\varinjlim\limits_{\CI_n} H_R(G, \CI_n)$.
\end{proof}

\subsection{} Now we explain how to deduce Theorem \ref{trace-p-w} from Theorem \ref{disc}. The argument is the same as in \cite{BDK}. We sketch it here for the convenience of the reader. 

We assume that the trace Paley-Wiener theorem holds for all proper Levi subgroups. Let $f \in R(G)^*_{\good}$. Then $f \mid_{A_R(\fkR_R(G))} \in A_R(\fkR_R(G))^*_{\good}$. By Theorem \ref{disc}, there exists $h \in \bar H_R$ such that $f-\Tr_R(h)$ vanishes on $A_R(\fkR_R(G))$. Therefore $A_R^*(f-\Tr_R(h))=0$ as a linear form on $\fkR_R(G)$. We have $A_R^* \Tr_R(h)=\Tr_R(\bar A^{\red}_R(h))$ as a linear form on $\fkR_R(G)$. Hence $$A_R^*(f)=a f-\sum_{M \text{ proper}} c_M r^*_{M, R} i^*_{M, R} (f) \in \Tr_R(\bar H_R).$$ By \cite[Proposition 3.2]{BDK}, $i_{M, R}^*(f)$ is a good form on $\fkR_R(M)$. By the inductive hypothesis on $M$, we have $$r^*_{M, R} i^*_{M, R} (f) \in r^*_{M, R} \fkR_R(M)^*_{\good}=r^*_{M, R} \Tr_R(\bar H_R(M)) \subset \Tr_R(\bar H_R).$$ Therefore $a f\in \Tr_R(\bar H_R)$ and $f \in \Tr_R(\bar H_R)$. 

\subsection{} We can now prove the rigid trace Paley-Wiener theorem. Define the rigid quotient of the Grothendieck group of representations
$$\fkR_R(G)_\rig=\fkR_R(G)/\fkR_R(G)_{\mathsf{diff}},$$ where 
$\fkR_R(G)_{\mathsf{diff}} \subset \fkR_R(G)$ is spanned by $i_M(\s)-i_M(\s \otimes \psi)$. Here $M$ ranges over the set of standard Levi subgroups, $\s\in\fkR_R(M)$ and $\psi$ is an unramified character of $M$ over $R$ which is trivial on $Z(G)^0$, the identity component of the center of $G$.

This definition is motivated by \cite{CH}, where the analogous notion for affine Hecke algebras was studied. For affine Hecke algebras, it is proved in \cite{CH} that for generic parameters, the trace map gives a perfect pairing between the the rigid cocenter of the affine Hecke algebras and the rigid quotient of Grothendieck group of representations. The advantage of considering the rigid cocenter and the rigid quotient instead of the whole cocenter and the Grothendieck group is that for affine Hecke algebras with semisimple root data, both the rigid cocenter and the rigid quotient are finite dimensional and the dimension can be computed explicitly. This allows us to have a good understanding of the relation between the cocenter and representations. 

We expect a similar phenomenon for the Hecke algebra and representations of $p$-adic groups. In the rest of this section, we establish the trace Paley-Wiener for the rigid cocenter. 

Let $(\fkR_R(G))^*_{\rig}$ be the $R$-linear functions on $\fkR_R(G)_\rig$ and $(\fkR_R(G))_{\rig, \good}^*=(\fkR_R(G))^*_\rig\cap (\fkR_R(G))^*_{\good}.$

\begin{proposition}
The trace map $\Tr_R: \bar H^{\rig, \red}_R \to (\fkR_R(G))_{\rig, \good}^*$ is surjective. 
\end{proposition}

\begin{proof}
Note first that $\fkR_R(G)_{\mathsf{diff}}$ vanished on $\bar H^\rig$. Hence the trace map induces a map $\bar H^\rig_R \to (\fkR_R(G))_{\rig, \good}^*$. 

Now let $f \in (\fkR_R(G))_{\rig, \good}^*$. By Theorem \ref{trace-p-w}, $f=\Tr_R(h)$ for some $h \in \bar H_R^{\red}$. We write $h$ as $h=\sum_{v \in V_+} h_v$, where $h_v \in \bar H(v)_R^{\red}$. 

By \cite[Theorem 4.1 and Theorem 6.1]{hecke-2}, for any $v \in V_+$, $\bar i_v: \bar H_R(M_v; v)^{\red} \to \bar H_R(G; v)^{\red}$ is bijective. Thus $h_v=\bar i_v(h'_v)$ for some $h'_v \in \bar H_R(M_v; v)^{\red}$. If $h_v \neq 0$, then $h'_v \neq 0$ as an element in $\bar H_R(M_v)^{\red}$. Thus by the definition of reduced cocenter, there exists $\s \in \fkR_R(M_v)$ such that $\Tr_R(h'_v, \s) \neq 0$. By Theorem \ref{cat-r} and the Mackey-type formula (Proposition \ref{mackey}), we have that $$\Tr_R(h_v, i_{M_v}(\sigma \circ \chi))=\Tr_R(\bar r_{M_v} \bar i_v(h'_v), \s \circ \chi)$$ is a regular function on $\chi$, with leading term $\Tr_R(h'_v, \s) \<\chi, v\>$. 

By the definition of $\fkR(G)_{\mathsf{diff}}$, we have that $h_v=0$ for any $v \in V_+$ such that $M_v \neq G$. Therefore $h \in \bar H^{\rig, \red}_R$. 
\end{proof}

\section{Application: the abstract Selberg principle}

\subsection{} Define $$\aleph_c=\{\nu=(\t, v) \in \aleph; \t \text{ is of finite order in } \Omega, M_v=G\}.$$ Let $G^c \subset G$ be the subset of compact elements. Then we have that $$G^c=\sqcup_{\nu \in \aleph_c} G(\nu).$$ Let $H^c=\oplus_{\nu \in \aleph_c} H(\nu)$ be the subset of $H$ consisting of functions supported in $G^c$ and $\bar H^c=\oplus_{\nu \in \aleph_c} \bar H(\nu)$ be the image of $H^c$ in $\bar H$. Note that by definition, the rigid cocenter $\bar H^\rig$ consists of elements in the cocenter represented by functions with support in the subset of compact-modulo-center elements of $G$. Thus $\bar H^c \subseteq \bar H^\rig$ and the equality holds if $G$ is semisimple. 

Let $\fkK_R(G)$ be the $R$-vector space with a basis given by the isomorphism classes of indecomposable finitely generated projective $G$-modules over $R$. Let $$\Rk_R: \fkK_R(G) \to \bar H_R$$ be the rank map, which sends a finitely generated projective module to the image in $\bar H$ of the trace of the idempotent of $M_n(H)$ defining it as a quotient of $H^n$. The main result of this section is the following abstract Selberg principle. 

\begin{theorem}\label{Sel}
The image of the rank map $\Rk_R$ is contained in $\bar H_R^c$. 
\end{theorem}

\begin{proof}
We argue by induction on the semisimple rank of $G$. 

If the semisimple rank of $G$ is $0$, then $G$ is compact modulo center. The statement is easy to prove in this case. 

Now we assume that the semisimple rank of $G$ is positive and that the statement holds for all the proper Levi subgroups of $G$ (which have smaller semisimple rank). Let $\varPi \in \fkK_R(G)$. By Frobenius reciprocity, we have that $r_{M, R}(\varPi) \in \fkK_R(M)$ for all Levi subgroup $M$. By the compatibility of restriction with rank map, we have $\Rk_M r_{M, R}(\varPi)=\bar r_{M, R} \Rk_G(\varPi) \in \bar H_R(M)$. By the inductive hypothesis, $\bar r_{M, R} \Rk_G(\varPi) \in \bar H_R(M)^c$ for all proper Levi $M$. 

We write $\Rk_G(\varPi)$ as $\Rk_G(\varPi)=\sum_{\nu \in G} f_{\nu}$, where $f_{\nu}$ is the corresponding Newton component of $\Rk_G(\varPi)$. Then by the Newton decomposition on $\bar H_R(M)$ and Proposition \ref{r-N}, if $\nu \notin \aleph_c$, then $\bar r_{M, R} f_{\nu}=0$ for all proper Levi $M$. Now by Theorem \ref{main'}, we have $f_{\nu}=0$ for $\nu \notin \aleph_c$. Therefore $\Rk_G(\varPi) \in \bar H_R^c$. 
\end{proof}

\subsection{} We make some comments about the abstract Selberg principle. The classical statement \cite{BB} of the abstract Selberg principle is as follows.

Let $f \in \bar H_\BC$ be the image of the rank map of a finitely generated projective representation of $G$ over $\BC$. Then the orbital integral of $f$, relative to a non-compact element of $G$ (in case $\text{char}(F)=0$) and a non-compact semisimple element of $G$ (in case $\text{char}(F)>0$), vanishes. 

A subtle issue here is that if $\text{char}(F)>0$, then it is not known in general whether the orbital integral of an unipotent element converges, although most cases are settled by McNinch in \cite{Mc}. 

The statement we have in this paper is somehow different. The statement says that the element $f \in \bar H_R$ (for arbitrary algebraically closed field $R$ of characteristic not equal to $p$) can be represented by a function on $G$ supported in compact elements. In characteristic $0$, or more generally, if the characteristic $l$ of $R$ is not in a certain finite set of primes $P$ \cite[Th\'eor\`eme C.1]{VW}, it is known that the following two conditions are equivalent: 

(1) An element $f$ of the cocenter is represented by a function supported in compact elements;

(2) The orbital integral of any regular semisimple non-compact element on $f$ vanishes. 

It is obvious that (1) implies (2).  The equivalence of (1) and (2) follows from the geometric density theorem, which is known under the assumption that $l\notin P$ \cite[Th\'eor\`eme C.2]{VW}. We remark that the exact set $P$ is not known in general, and that it is not sufficient to assume that $l$ is banal in the sense of \cite[II.3.9]{Vig}. This is related to the question of characterization of the primes that divide the Assem number of orbital integrals, see \cite[section B]{VW}.

\appendix 

\section{A different proof in characteristic zero}\label{appendix}

In the appendix, we give a different proof that the elliptic reduced cocenter is contained in the rigid reduced cocenter, under the assumption that $F$ is of characteristic $0$. % and $\sharp W_G$ is invertible in $R$. 
The proof is based on Clozel's integration formula \cite[Proposition 1]{clozel}.

\subsection{} We recall several foundational results regarding characters of admissible representations motivated by \cite[Proposition 1]{clozel}. Let $G_{\text{rs}}$ denote the set of regular semisimple elements of $G$. A (weak) $R$-analogue of the classical result of Harish-Chandra, see \cite[E.3.4.4]{VW}, says that for each admissible $G$-representation $(\pi,V)$, there exists a locally constant function 
\[\tr_V: G_{\text{rs}}\to R,
\]
the character of $(\pi,V)$, such that, for every $f\in \bar H$,
\[\Tr_R(f,V)=\int_{G_{rs}} \tr_V(g) f(g) dg.
\]
For every semisimple element $s\in G$, recall the parabolic (Deligne's parabolic) $P_s=M_s N_s$ contracted by $s$, see \cite{deligne,cass}.
Let $G_c$ be the ``compact part'' of $G$ as defined in \cite[\S 1]{clozel}, i.e. the set of semisimple elements $s$ such that $P_s=G$. Then $G_c$ is the set of semisimple elements in $G^\rig$. We first have the generalization of Casselman's formula \cite{cass}, see \cite[II.3.7]{Vig} and \cite[Theorem 7.4]{MS}:

\begin{theorem}For every admissible $G$-representation $(\pi,V)$ and $s\in G_{\text{rs}}$, $$\tr_V(s)=\tr_{r_{M_s}(\pi)}(s).$$ 
\end{theorem}

\subsection{} Let $T\subset G$ be a Cartan subgroup and let $W(G,T)$ is the Weyl group of $G$ with respect to $T$. For every $t\in T\cap G_{\text{rs}}$ and $f\in \bar H$, let 
\begin{equation}
\mathcal O_t(f)=\int_{G/T} f(x t x^{-1}) d(xT)
\end{equation}
denote the orbital integral. Suppose $T\subset M$. Then for $t\in T\cap G_{\text{rs}}$, we have the known descent formula for orbital integrals:
\begin{equation}
\mathcal O_t(f)=\mathcal O_t^M(\bar r_M(f)),
\end{equation}
where $\mathcal O_t^M$ is the orbital integral taken in $M$.

We will use a variant of  Weyl's integration formula as in \cite[page 241]{clozel}. Since $T\cap G_{\text{rs}}$ is totally disconnected with a free action of $W(G,T)$, one may find an open and closed subset $T^0_{\text{rs}}$ of  $T\cap G_{\text{rs}}$ such that $T\cap G_{\text{rs}}=\bigsqcup_{w\in W(G,T)} wT^0_{\text{rs}}$. Then the morphism
\[G/T\times T^0_{\text{rs}}\to G_{\text{rs}},\ (gT,t)\mapsto \text{Ad}(g)t
\]
is one-to-one. This is because $\text{Ad}(g)t=t'$ implies $\text{Ad}(g)T=T$ hence there exists $w\in W(G,T)$ such that $\text{Ad}(g)=w$ as homomorphisms of $T$. But since $T^0_{\text{rs}}$ is a fundamental domain for the action of $W(G,T)$, it follows that $w=1$. Then the form of Weyl's integration formula for a locally constant, compactly supported function $F$ takes the form:
\[
\int_G F(g) dg=\sum_{T\subset G/_\sim} \int_{T^0_{\text{rs}}} |D_G(t)| \int_{G/T} F(x t x^{-1}) d(xT) dt.
\]
Here $T$ runs over a set of representatives of Cartan subgroups of $G$ (modulo $G$-conjugation) and $D_G(t)$ is the coefficient of $x^{\text{rank}(G)}$ in $\det(x+1-\text{Ad}(t))$. 
 %as long as $\sharp W_G$ is invertible in $R$, 
If $f\in \bar H_R$, specialize $F=f\cdot \tr_V$, which is compactly-supported locally constant (since $\tr_V$ is locally constant on $G_{\text{rs}}$), and arrive at the formula  \cite[page 241]{clozel}:
\begin{equation}\label{Weyl}
%\Tr_R(f,V)=\sum_{T\subset G/_\sim}\frac 1{|W(G,T)|} \int_{T\cap G_{\text{rs}}} |D_G(t)| \tr_V(t) \mathcal O_t(f) dt.
\Tr_R(f,V)=\sum_{T\subset G/_\sim} \int_{T^0_{\text{rs}}} |D_G(t)| \tr_V(t) \mathcal O_t(f) dt.
\end{equation}
Following \cite[page 240]{clozel}, denote 
\[\Tr_{R, c}(f,V)=\int_{G_c\cap G_{\text{rs}}} f(g) \tr_V(g) dg.
\]

Now we can show that the elements of $\bar A_R^{\cat}(\bar H)$ are essentially supported on $G^{\text{rig}}$.

\begin{proposition}\label{p-rig} For every $f\in\bar A_R^{\cat}(\bar H)$ and $(\pi,V)$ a smooth admissible $G$-representation, 
\[\Tr_R(f,V)=\Tr_{R, c}(f,V).\]
\end{proposition}

\begin{proof} 
Following the proof of \cite[Proposition 1]{clozel}, rewrite (\ref{Weyl}) according to the Deligne stratification of Cartan subgroups $T=\cup S$, where $t,t'\in T$ are in the same stratum $S$ if $P_t=P_{t'}$. In this case, write $P_S=P_t$ for $t\in S$:
\[ \Tr_R(f,V)=\sum_{S\subset G/_\sim} I(S),\quad I(S)=\int_{S\cap G_{\text{rs}}/_\sim} |D_G(t)| \tr_V(t) \mathcal O_t(f) dt.
\]
Since $f\in\bar A_R^{\cat}(\bar H),$ we have $\bar r_M(f)=0$ for all proper Levi subgroups $M$ of $G$. The descent formula for orbital integrals implies therefore that if $T\subset M$, then $\mathcal  O_t(f)=0$ for all $t\in T\cap G_{\text{rs}}.$ This means that
\[ \Tr_R(f,V)=\sum_{S\subset G_c/_\sim} I(S),\]
and for each such $S$, $P_S=G$. The claim follows.
\end{proof}

\subsection{} Now we give an alternative proof that $\bar H_R^{\el, \red} \subset \bar H_R^{\rig, \red}$ under the assumption that $F$ is of characteristic $0$.
% and $\sharp W_G$ is invertible in $R$. 

Let $f \in \bar H^{\el, \red}$. %By Proposition \ref{im-A}, $f \in \bar A_R^{\red}(\bar H)$. 
Recall that $\bar H^{\rig, \red}_R=\bigoplus_{v \in V_+; M_v=G}\bar H^{\red}_R(v)$ and set  $\bar H^{\text{nrig}, \red}_R=\bigoplus_{v \in V_+; M_v\neq G}\bar H^{\red}_R(v)$. Using Theorem \ref{t:newton}, we write $f$ as $f=f_\rig+f_{\text{nrig}}$, where $f_\rig \in \bar H^{\rig, \red}$ and $f_{\text{nrig}} \in \bar H^{\text{nrig}, \red}$.  Since $G_c \subset G^\rig$, we have \begin{align*} \Tr_R(f, V) & = \Tr_{R, c}(f,V)= \Tr_{R, c}(f_\rig,V)+\Tr_{R, c}(f_{\text{nrig}},V)\\
&=\Tr_{R, c}(f_\rig,V)=\Tr_R(f_\rig,V).
\end{align*} 
Here the first equality follows from Proposition \ref{p-rig}, the second equality follows from the fact that the support of $f_{\text{nrig}}$ does not intersects $G_c$ and the third equality follows from the fact that the support of $f_\rig$ is contained in $G^\rig$. 

Therefore $f-f_\rig \in \ker \Tr_R$.

\end{document}